\documentclass[11pt]{amsart}
\usepackage{amsfonts,amsmath,amsthm,amssymb}
\usepackage{xcolor}
\usepackage{enumerate}
\usepackage{comment}
\usepackage{mathtools}
\usepackage{mdframed}

\newmdenv[
  backgroundcolor=yellow!10,
  linecolor=orange,
  linewidth=2pt,
  topline=false,
  bottomline=false,
  rightline=false,
  leftline=true,
  skipabove=10pt,
  skipbelow=10pt,
  innerleftmargin=10pt,
  innerrightmargin=10pt,
  innertopmargin=5pt,
  innerbottommargin=5pt,
]{comentario}

\newtheorem*{questionC}{\textbf{Problem C}}
\newtheorem*{questionCinZ}{\textbf{Problem C in $\mathbb{Z}_{\rho}[[z_1,\ldots,z_m]]$}}

\newtheorem{theorem}{\textbf{Theorem}}

\newtheorem{corollary}{\textbf{Corollary}}
\newtheorem{remark}{\textbf{Remark}}

\newtheorem{lemma}{\textbf{Lemma}}

\def\A {\mathbb{A}}

\def\N {\mathbb{N}}
\def\Z {\mathbb{Z}}

\def\Q {\mathbb{Q}}
\def\R {\mathbb{R}}

\def\QQ {\overline{\Q}}
\def\C {\mathbb{C}}
\def\K {\mathbb{K}}

\theoremstyle{remark}

\numberwithin{equation}{section}

\begin{document}
	
	\title[Exceptional Sets of Transcendental Analytic Functions]{On the Exceptional Sets of Transcendental Analytic Functions in Several Variables with Integer Coefficients}

	\author[J. Lelis]{Jean Lelis}
	\address{Faculdade de Matemática/ICEN/UFPA, Belém - PA, Brazil.}
	\email{jeanlelis@ufpa.br}

         \author[B. De Paula Miranda]{Bruno De Paula Miranda}
	\address{Instituto Federal de Goi\'{a}s, Avenida Saia Velha, Km 6, BR-040, s/n, Parque Esplanada V, Valpara\'{i}so de Goi\'{a}s, GO 72876-601, Brazil}
	\email{bruno.miranda@ifg.edu.br}
	
	\author[C. G. MOREIRA]{CARLOS GUSTAVO MOREIRA}
\address{Instituto de Matemática Pura e Aplicada IMPA, Rio de Janeiro-RJ, Brasil.}
\email{gugu@impa.br}

	
	
	\begin{abstract}
In 2020, Marques and Moreira proved that every subset of $\QQ \cap B(0,1)$, which is closed under complex conjugation and contains $0$, is the exceptional set of uncountably many transcendental analytic functions with integer coefficients. In this paper, we extend this result to transcendental analytic functions in several variables. In particular, we show that every subset of $\QQ^m$ contained in the unit polydisc $\Delta(0;1)$, closed under complex conjugation and containing the zero vector, is the exceptional set of uncountably many transcendental analytic functions in several variables with integer coefficients.
\end{abstract}
	
	\maketitle
	
	\section{Introduction}

A \textit{transcendental} analytic function in $m$ complex variables is an analytic function $f:\Omega\subseteq\C^{m}\to\C$ such that the only complex polynomial $P(X_1,\ldots,X_{m+1})\in\C[X_1,\ldots,X_{m+1}]$ satisfying 
\[
P(z_1,\ldots,z_m,f(z_1,\ldots,z_m))=0 
\]
for all $(z_1,\ldots,z_m)\in\Omega$ is the zero polynomial. The study of the arithmetic behavior of transcendental analytic functions in one complex variable can be traced back to an 1886 letter from Weierstrass to Strauss, in which Weierstrass proved the existence of such functions that map $\Q$ into itself. Motivated by this kind of problem, he defined the \textit{exceptional set} of an analytic function in one complex variable $f:\Omega\subseteq\C \to \C$ as
\[
S_f := \{\alpha\in \QQ\cap \Omega : f(\alpha)\in \QQ\}.
\]
Therein, Weierstrass conjectured the existence of a transcendental entire function whose exceptional set is $\QQ$. This conjecture was confirmed in 1895 by St\"ackel \cite{stackel1895}.

In his book \cite{mahler1976}, Mahler studied various problems related to exceptional sets of transcendental functions, particularly those with coefficients belonging to a fixed subset of the complex numbers. Among these, he formulated three problems, the third of them is the following:

\begin{questionC}
For any choice of $S\subseteq \QQ \cap B(0, \rho)$ that is closed under complex conjugation and contains 0, where $\rho \in (0, \infty]$, does there exist a transcendental analytic function $f$ with rational coefficients and radius of convergence $\rho$ such that $S_f = S$?
\end{questionC}

In 2016, Marques and Ramirez \cite{ramirez2016} answered this question affirmatively when $\rho = \infty$ (i.e., for entire functions). Their result was later extended by Marques and Moreira in \cite{gugu2018}, who provided an affirmative answer to Mahler's Problem {\sc C} for any $\rho \in (0, \infty]$.

Let $f:\Omega\subseteq\C^m\to\C$ be an analytic function in $m$ complex variables. We define its exceptional set as
\[
S_f := \{(\alpha_1,\ldots,\alpha_m)\in\QQ^m \cap \Omega : f(\alpha_1,\ldots,\alpha_m)\in\QQ\}.
\]

In \cite{lelis1}, Alves et al. solved the analogue of Mahler's  problem {\sc C} for functions of several variables. They proved that for any subset $S \subseteq \QQ^m$ that is closed under complex conjugation and contains the zero vector, there exist uncountably many transcendental entire functions in $m$ complex variables with rational coefficients whose exceptional set is $S$. Moreover, it follows immediately from the result proved in \cite{gugu2018} that, given any polydisc centered at the origin, one can restrict the domain of analyticity of such functions to this fixed polydisc.

In this paper, we consider Mahler's Problem C for transcendental analytic functions in several variables with integer coefficients. To clarify the exposition, we recall some definitions. Given $w = (w_1, \ldots, w_m) \in \C^m$ and $\rho = (\rho_1, \ldots, \rho_m)\in\R_{>0}^m$, we define the \textit{open polydisc} centered at $w$ with \textit{polyradius} $\rho$ as the set
\[
\Delta(w;\rho) := \{ z = (z_1, \ldots, z_m) \in \C^m : |z_i - w_i| < \rho_i \ \forall \ 1 \leq i \leq m \}.
\] When $\rho = (1, \ldots, 1)$ we write $\Delta(w; 1)$. Polydiscs centered at the origin, which we denote by $\Delta(0;\rho)$, are the natural regions of convergence for power series in several variables.

Let $\K_{\rho}[[z_1, \ldots, z_m]]$ denote the set of power series
\[
f(z_1,\ldots, z_m) = \sum_{(k_1,\ldots,k_m) \in \N_0^m} f_{(k_1,\ldots,k_m)} z_1^{k_1}\cdots z_m^{k_m},
\]
with coefficients in $\K$ and polyradius of convergence $\rho=(\rho_1,\ldots,\rho_m)\in\R_{>0}^m$, where $\N_0 = \{0\} \cup \N$. We denote by $\A_{\rho}^m$ the set of algebraic numbers in the polydisc $\Delta(0; \rho)$, and by $\A^m$ the set of algebraic numbers in $\Delta(0;1)$, the unit polydisc. Moreover, we write $\Z\{z_1, \ldots, z_m\}$ for the set of power series with integer coefficients and polyradius of convergence $\rho = (1,\ldots,1)$. With this notation, our main objective in this paper is to study the following multivariable version of Mahler’s Problem {\sc C}.

\begin{questionCinZ}
Let $m$ be a positive integer and let $\rho$ be an $m$-tuple in $(0,1]^m$. Does there exist, for any choice of $S \subseteq \A_{\rho}^m$ closed under complex conjugation and containing the zero vector $(0, \ldots, 0)$, a transcendental function $f \in \Z_{\rho}[[z_1, \ldots, z_m]]$ for which $S_f = S$?
\end{questionCinZ}

Note that the conditions of being closed under complex conjugation and containing the zero vector are necessary, since the power series has integer coefficients. In 2020, Marques and Moreira \cite{gugu20} answered Mahler’s Problem {\sc C} in $\Z\{z\}$ affirmatively. Specifically, they proved that any subset of $\A$ (closed under complex conjugation and containing $0$) is the exceptional set of a transcendental analytic function in $\Z\{z\}$. Our goal is to generalize this result to several variables. More precisely, we prove the following result.

\begin{theorem}[Main Theorem]\label{maintheo}
Let $m$ be a positive integer and let $\rho$ be an $m$-tuple in $(0,1]^m$. Then, every subset of $\A_{\rho}^m$ that is closed under complex conjugation and contains the zero vector is the exceptional set of uncountably many transcendental functions in $\Z_{\rho}[[z_1, \ldots, z_m]]$.
\end{theorem}

Our proof relies on two key results. The first is a multivariable extension of a classical lemma by Lekkerkerker \cite{lekk}, which concerns interpolation over finite sets using power series in one variable with bounded integer coefficients. The second is a multivariable adaptation of a lemma from \cite{gugu20}, which guarantees that for any $\alpha \in \A^m$, there exists a power series $f \in \Z\{z_1, \ldots, z_m\}$ such that $f(z) = 0$ for $z \in \A^m$ if and only if $z \in \{\alpha, \overline{\alpha}\}$. In the next section, we prove these lemmas; in the following section, we establish our main theorem.

\section{Key Lemmas}

In this section, we shall provide some lemmas which are essential ingredients in our proof (they are theoretical results which may be of some interest). The proof of the Lemma \ref{lem1} is based on the analogous result of Lekkerkerker \cite{lekk} for one variable.

\begin{lemma}\label{lem1}
    Let $m$ be a positive integer, and let $(\alpha_1,\beta_1),\ldots,(\alpha_n,\beta_n)$ be finitely many pairs in $\Delta(0;1) \times \C$ satisfying the following conditions:
    \begin{enumerate}[(i)]
        \item The points $\alpha_1, \ldots, \alpha_n$ are distinct and non-zero.
        \item If $\alpha_k$ is a real, then $\beta_k$ is real.
        \item If $\alpha_k$ is non-real, there exists an index $\ell$ such that $\alpha_\ell$ and $\beta_\ell$ are the complex conjugates of $\alpha_k$ and $\beta_k$, respectively.
    \end{enumerate}
    Then there exists a power series $f \in \Z\{z_1, \ldots, z_m\}$ with bounded coefficients such that 
    \[f(\alpha_k) = \beta_k\]
    for all $1 \leq k \leq n$.
\end{lemma}

\begin{proof}
When $m = 1$, this is precisely the result proved by Lekkerkerker in \cite{lekk}. Therefore, we assume $m\geq 2$. Let $(\alpha_1, \beta_1), \ldots, (\alpha_n, \beta_n)$ be pairs in $\Delta(0;1) \times \C$ satisfying the assumptions of the lemma. Then there exist complex hyperplanes $H_k$, defined as the zero sets of the polynomials
\[
h_k(z_1,\ldots,z_m)=a_{k,1} z_1 + \cdots+b_{k,m} z_m + c_k,
\]
for $1\leq k\leq n$ such that
\begin{enumerate}[(a)]
    \item $H_k$ does not contain the origin, and $\alpha_j \in H_k$ if and only if $j = k$;
    \item If $\alpha_k$ is real, then $h_k$ has real coefficients;
    \item If $\alpha_\ell = \overline{\alpha_k}$, then $h_\ell = \overline{h_k}$.
\end{enumerate}

Using these hyperplanes, we define the following polynomials
\[
q(z) = \prod_{k=1}^n h_k(z), \quad
q'(z) = \sum_{k=1}^n \prod_{\substack{j=1 \\ j \neq k}}^n h_j(z), \quad
p(z) = \sum_{k=1}^n \beta_k \frac{q(z)}{q'(\alpha_k) h_k(z)},
\]
where $z = (z_1,\ldots,z_m)$, and the expression is well-defined because $q'(\alpha_k) \neq 0$ for all $k$ by construction. From the symmetry conditions on the $H_k$, it follows that $p$ and $q$ have real coefficients, and they satisfy
\[
q(\alpha_k) = 0, \quad p(\alpha_k) = \beta_k \quad \text{for all } 1 \leq k \leq n.
\]

We now define the desired power series $f \in \Z\{z_1,\ldots,z_m\}$ as
\begin{equation}\label{eqf}
f(z) = p(z) + q(z)g(z),
\end{equation}
where $g \in \R[[z_1,\ldots,z_m]]$ will be chosen so that $f$ has bounded integer coefficients. Writing
\[
g(z) = \sum_{\theta_k \in \N_0^m} g_{\theta_k} z^{\theta_k}, \quad
f(z) = \sum_{\theta_{k} \in \N_0^m} f_{\theta_k} z^{\theta_k},
\]
and similarly expanding
\[
p(z) = \sum_{|\theta_k| \leq n-1} a_{\theta_k} z^{\theta_k}, \quad
q(z) = \sum_{|\theta_k| \leq n} b_{\theta_k} z^{\theta_k},
\]
where $|\theta_k|=t_{1,k}+\cdots+t_{m,k}$ for $\theta_k=(t_{1,k},\ldots,t_{m,k})\in\N_0^m$. We substitute into \eqref{eqf} and obtain recursive equations for $f_{\theta_k}$ in terms of $a_{\theta_k}$, $b_{\theta_k}$, and $g_{\theta_k}$ given by
\[
f_{\theta_k} = a_{\theta_k}+g_{\theta_k} b_0 + \sum_{\phi+\sigma=\theta_k} g_{\phi} b_{\sigma},
\]
where $b_0 = b_{(0,\ldots,0)} \neq 0$ by definition of $q(z)$ and we consider $a_{\theta_k} = 0$ for $|\theta_k| \geq n$.

Therefore, We can recursively choose the coefficients $g_{\theta_k} \in \R$ such that $f_{\theta_k} \in \Z$ and
\[
-\frac{1}{2|b_0|} \leq g_{\theta_k} \leq \frac{1}{2|b_0|} \quad \text{for all } \theta_k \in \N_0^m.
\]
This ensures that all coefficients $f_{\theta_k}$ are integers and uniformly bounded
\[
|f_{\theta_k}| \leq |\alpha_{\theta_k}|+ \left( \sum_{\phi+\sigma=\theta_k} |b_{\sigma}| \right) \cdot \max_{\phi} |g_{\phi}| \leq |\alpha_{\theta_k}|+ \frac{L(q)}{2|b_0|},
\]
where $L(q)$ denotes the length (i.e., the sum of absolute values of the coefficients) of the polynomial $q$.

It follows that $f$ has bounded coefficients, and the same holds for $g$. Hence, both $f$ and $g$ converge on the unit polydisc $\Delta(0;1)$. Moreover, by construction,
\[
f(\alpha_k) = \beta_k \quad \text{for all } 1 \leq k \leq n.
\]
This concludes the proof.
\end{proof}

\begin{remark}\label{rem1}
In the construction of the function $f$ in Lemma \ref{lem1}, it is worth noting that the constant term $f_0 = f_{(0,\ldots,0)}$ can be chosen to be any prescribed integer, by adjusting the bound on the coefficients of $f$ if necessary. Indeed, it suffices to select the coefficient $g_0 = g_{(0,\ldots,0)}$ so that
\[
f_0 = a_0 + g_0 b_0
\]
matches the desired value. This flexibility will be used in the proof of the next lemma.
\end{remark}

The following lemma is a multivariable version of Lemma 2 in \cite{gugu20}. In particular, it also holds in the univariate case for analytic functions with integer coefficients, and we shall use this fact in the proof for the multivariable case. Throughout, we consider the maximum norm on $\C^m$.

\begin{lemma}\label{lem2}
    Let $m$ be a positive integer and $\alpha \in \A^m$. Then there exists a function $f \in \Z\{z_1,\ldots,z_m\}$ such that $f(z) = 0$ for $z \in \A^m$ if and only if $z \in \{\alpha,\overline{\alpha}\}$. Moreover, there exists a positive constant $C$ depending only on $\alpha$ such that $|f(z)| \leq C/(1-\|z\|)^4$ for all $z \in \Delta(0;1)$.
\end{lemma}

\begin{proof}
    Suppose that $\alpha = (\alpha_1, \ldots, \alpha_m) \in \A^m$, and consider the set
    \[
    H(\alpha) := \{(z_1,\ldots,z_m) \in \C^m : z_i \in \{0,\alpha_i,\overline{\alpha}_i\} \quad \text{for} \quad i=1,\ldots,m\}.
    \]
    By Lemma \ref{lem1} and Remark \ref{rem1}, there exists a function $g \in \Z\{z_1,\ldots,z_m\}$ with bounded coefficients such that $g(\alpha) = g(\overline{\alpha}) = 0$ and $g(z) = 1$ for all $z \in H(\alpha) \setminus \{\alpha, \overline{\alpha}\}$. In particular, there exists a positive constant $\tilde{C}$ depending only on $\alpha$ such that
    \[
    |g(z_1,\ldots,z_m)| \leq \tilde{C}/(1-\|(z_1,\ldots,z_m)\|).
    \]
    
    Moreover, by the one-variable version proved in Lemma 2 of \cite{gugu20}, there exist functions $h_i(z) \in \Z\{z\}$ such that $h_i(z) = 0$ for $z \in \A$ if and only if $z \in \{\alpha_i,\overline{\alpha}_i\}$, for all integers $1 \leq i \leq m$. Furthermore,
    \begin{equation}\label{C_i}
    |h_i(z)| \leq C_i/(1-|z|)^3,
    \end{equation}
    where $C_i$ is a constant depending only on $\alpha_i$. 

    Now, consider an enumeration of $\A^m \setminus H(\alpha)$ given by
    \begin{equation}\label{enum2}
        \A^m \setminus H(\alpha) = \{\beta_1, \beta_2, \ldots\} \cup \{\overline{\beta_1}, \overline{\beta_2}, \ldots\}.
    \end{equation}
    We construct a power series $f\in\Z\{z_1,\ldots,z_m\}$ defined by
    \begin{equation}\label{fnull}
        f(z_1, \ldots, z_m) = g(z_1,\ldots,z_m) + \sum_{i=1}^m \sum_{j=1}^\infty h_i(z_i) z_i^{t_{i,j}},
    \end{equation}
    such that $f(z)=0$ with $z\in\A^m$ if and only if $z\in\{\alpha,\overline\alpha\}$, where the sequences $(t_{i,j})_{j \geq 1}$ are strictly increasing sequences of positive integers for each $1 \leq i \leq m$, defined inductively using the enumeration \eqref{enum2}.
    
    To define  the sequences $(t_{i,j})_{j \geq 1}$ for each $1 \leq i \leq m$, let $\beta_j = (b_{1,j}, \ldots, b_{m,j})$ be the $j$-th term of the enumeration \eqref{enum2}, and set $b_j \coloneqq \max_{1 \leq i \leq m} |b_{i,j}|$. Since $\beta_1 \notin H(\alpha)$, there exist positive integers $t_{1,1}, \ldots, t_{m,1}$ such that
    \[
    B_1 \coloneqq g(\beta_1) + \sum_{i=1}^m h_i(b_{i,1}) b_{i,1}^{t_{i,1}} \neq 0.
    \]
    Moreover, since $0 < b_1 < 1$, there exists a positive integer $s_1$ such that
    \[
    m \hat{C} \frac{b_1^{s_1}}{(1-b_1)^4} < \frac{|B_1|}{2},
    \]
    where $\hat{C} = \max_{1 \leq i \leq m}\{C_i\} $ (see \eqref{C_i}).

    Similarly, since $\beta_2 \notin H(\alpha)$ and $0 < b_2 < 1$, we can choose integers $t_{1,2}, \ldots, t_{m,2}$ such that $t_{i,2} \geq \max\{t_{i,1}+1, s_1\}$ for all $1 \leq i \leq m$, and
    \[
    B_2 \coloneqq g(\beta_2) + \sum_{i=1}^m \sum_{j=1}^2 h_i(b_{i,2}) b_{i,2}^{t_{i,j}} \neq 0.
    \]
    Additionally, there exists a positive integer $s_2$ such that
    \[
    m \hat{C} \frac{b_2^{s_2}}{(1-b_2)^4} < \frac{|B_2|}{2}.
    \]

    Proceeding inductively, suppose we have defined positive integers $t_{i,j}$ for all $1 \leq j \leq n-1$ and $1 \leq i \leq m$, such that $t_{i,j+1} \geq \max\{t_{i,j}+1, s_j\}$ and
    \[
    B_k \coloneqq g(\beta_k) + \sum_{i=1}^m \sum_{j=1}^k h_i(b_{i,k}) b_{i,k}^{t_{i,j}} \neq 0,
    \]
    with $s_k$ satisfying
    \[
    m \hat{C} \frac{b_k^{s_k}}{(1-b_k)^4} < \frac{|B_k|}{2},
    \]
    for all $1 \leq k \leq n-1$.
    
    Since $\beta_n \notin H(\alpha)$, there again exist positive integers $t_{1,n}, \ldots, t_{m,n}$ such that $t_{i,n} \geq \max\{t_{i,n-1}+1, s_{n-1}\}$ and
    \[
    B_n = g(\beta_n) + \sum_{i=1}^m \sum_{j=1}^{n} h_i(b_{i,n}) b_{i,n}^{t_{i,j}} \neq 0.
    \]
    Again, there exists a positive integer $s_n$ such that 
      \[
    m \hat{C} \frac{b_n^{s_n}}{(1-b_n)^4} < \frac{|B_n|}{2}.
    \]
    Thus, we inductively construct strictly increasing sequences $(t_{i,j})_{j \geq 1}$ such that $B_j \neq 0$ and $t_{i,j} \geq s_{j-1}$ for all integers $1 \leq i \leq m$ and $j \geq 1$. Given these sequences, the power series defined by \eqref{fnull} satisfies
    \begin{align*}
        |f(z_1,\ldots,z_m)| &= \left| g(z_1,\ldots,z_m) + \sum_{i=1}^m \sum_{j=1}^\infty h_i(z_i) z_i^{t_{i,j}} \right| \\
                            &\leq |g(z_1,\ldots,z_m)| + \sum_{i=1}^m \sum_{j=1}^\infty |h_i(z_i)||z_i|^{t_{i,j}} \\
                            &\leq \frac{\tilde{C}}{(1-\|(z_1,\ldots,z_m)\|)} + \frac{m\hat{C}}{(1-\|(z_1,\ldots,z_m)\|)^4} \\
                            &\leq \frac{C}{(1-\|(z_1,\ldots,z_m)\|)^4},
    \end{align*}
    for all $(z_1,\ldots,z_m) \in \Delta(0;1)$, where $C = \tilde{C} + m\hat{C}$ is a positive constant depending only on $\alpha$. In particular, we conclude that $f \in \Z\{z_1,\ldots,z_m\}$.
     
    It remains to verify that \eqref{fnull} satisfies $f(z)=0$ with $z\in\A^m$ if and only if $z\in\{\alpha,\overline\alpha\}$. Indeed, since
    \[
    (h_1(z_1)z_1,\ldots,h_m(z_m)z_m) = (0,\ldots,0)
    \]
    for all $(z_1,\ldots,z_m) \in H(\alpha)$, it follows that $f(z) = g(z)$ for all $z \in H(\alpha)$. In particular, $f(z) = 0$ for $z \in H(\alpha)$ if and only if $z \in \{\alpha,\overline{\alpha}\}$.
    
    Now, suppose that $\beta \notin H(\alpha)$. Then, there exists a positive integer $n$ such that $\beta \in \{\beta_n, \overline{\beta}_n\}$, where $\beta_n$ is the $n$-th term of the enumeration \eqref{enum2}. Thus,
    \[
    f(\beta_n) = B_n + \sum_{i=1}^m \sum_{j = n+1}^\infty h_i(b_{i,n}) b_{i,n}^{t_{i,j}},
    \]
    with $t_{i,n+1} \geq s_n$ for all $1 \leq i \leq m$. Therefore,
    \begin{align*}
        \left| \sum_{i=1}^m \sum_{j = n+1}^\infty h_i(b_{i,n}) b_{i,n}^{t_{i,j}} \right| &\leq \sum_{i=1}^m \sum_{j = n+1}^\infty |h_i(b_{i,n})||b_{i,n}|^{t_{i,j}} \\
        &\leq \frac{\hat{C}}{(1-b_n)^3} \sum_{i=1}^m \sum_{t = s_n}^\infty b_n^t \\
        &\leq m\hat{C} \frac{b_n^{s_n}}{(1-b_n)^4} < \frac{|B_n|}{2},
    \end{align*}
    and hence
    \[
    |f(\beta_n)| \geq |B_n| - \frac{|B_n|}{2} > 0.
    \]
    Since $f(\overline{\beta}_n) = \overline{f(\beta_n)}$, it follows that $f(\overline\beta_n) \neq 0$ as well. Thus, $f(z) = 0$ for $z \in \A^m$ if and only if $z \in \{\alpha, \overline{\alpha}\}$. This completes the proof.
\end{proof}

We now present a simple lemma that will be important for controlling the radius of convergence of the functions we aim to construct with integer coefficients.

\begin{lemma}\label{lem3}
Let $m$ be a positive integer, and let $\rho = (\rho_1,\ldots,\rho_m)$ be an $m$-tuple in $(0,1]^m$. Then there exists a power series $g \in \Z_{\rho}[[z_1,\ldots,z_m]]$.
\end{lemma}

\begin{proof}
For each integer $1 \leq i \leq m$, define the sequence $\{a_{i,n}\}_{n \geq 0}$ by
\[
a_{i,n} := \left\lfloor \left(\frac{1}{\rho_i}\right)^n \right\rceil,
\]
for all $n \geq 0$, where $\lfloor x\rceil$ is the integer closest to $x\in\R$. Since $\rho_i^{-n} \leq a_{i,n} \leq \rho_i^{-n} + 1$, it is straightforward to verify that the power series
\[
g(z_1,\ldots,z_m) = \sum_{i=1}^m \sum_{j=1}^\infty a_{i,j} z_i^j
\]
has polyradius of convergence equal to $\rho = (\rho_1, \ldots, \rho_m)$.
\end{proof}

The following lemma will be used to show that the set of power series with integer coefficients that are algebraic (i.e., not transcendental) is countable. This result will be applied to ensure that uncountably many of the power series we construct in the proof of the theorem are, in fact, transcendental.

\begin{lemma}
Let $m$ be a positive integer, let $\rho$ be an $m$-tuple in $(0,1]^m$, and let $\mathbb{K}$ be a subfield of $\mathbb{C}$. Suppose that $f$ is a power series in $\mathbb{K}_{\rho}[[z_1,\ldots,z_m]]$ such that
\[
P(z_1,\ldots,z_m, f(z_1,\ldots,z_m)) = 0 \quad \text{for all } (z_1,\ldots,z_m) \in \Delta(0;\rho),
\]
where $P \in \mathbb{C}[X_1,\ldots,X_m,Y]$ is a nonzero polynomial of degree $n$. Then there exists a nonzero polynomial $\tilde{P} \in \mathbb{K}[X_1,\ldots,X_m,Y]$ of degree at most $n$ such that
\[
\tilde{P}(z_1,\ldots,z_m, f(z_1,\ldots,z_m)) = 0 \quad \text{for all } (z_1,\ldots,z_m) \in \Delta(0;\rho).
\]
\end{lemma}

\begin{proof}
By hypothesis, $f(z)$ is given by
\[
f(z) = \sum_{\theta_k \in \N_0^m} a_{\theta_k} z^{\theta_k} \quad \text{for all } z \in \Delta(0;\rho),
\]
where $a_{\theta_k} \in \mathbb{K}$ and $z^{\theta_k} = z_1^{t_{1,k}} \cdots z_m^{t_{m,k}}$ for all $\theta_k = (t_{1,k}, \ldots, t_{m,k}) \in \N_0^m$. Moreover, we write $P \in \C[X_1,\ldots,X_m,Y]$ as
\[
P(X) = \sum_{|\sigma_\ell|\leq n} C_{\sigma_\ell} X_1^{s_{1,\ell}} \cdots X_m^{s_{m,\ell}} Y^{s_{m+1,\ell}},
\]
where the summation is taken over all $\sigma_\ell = (s_{1,\ell}, \ldots, s_{m+1,\ell}) \in \N_0^{m+1}$ such that $|\sigma_{\ell}| = s_{1,\ell} + \cdots + s_{m+1,\ell} \leq n$. Then, by hypothesis,
\begin{align*}
0 &= P(z_1,\ldots,z_m,f(z_1,\ldots,z_m)) \\
  &= \sum_{|\sigma_\ell|\leq n} C_{\sigma_\ell} z_1^{s_{1,\ell}} \cdots z_m^{s_{m,\ell}} \left( \sum_{\theta_k \in \N_0^m} a_{\theta_k} z^{\theta_k} \right)^{s_{m+1,\ell}} \\
  &= \sum_{|\sigma_\ell|\leq n} C_{\sigma_\ell} z_1^{s_{1,\ell}} \cdots z_m^{s_{m,\ell}} \sum_{\nu_{k,\ell} \in \N_0^m} b_{\nu_{k,\ell}} z^{\nu_{k,\ell}}
\end{align*}
for all $z = (z_1,\ldots,z_m) \in \Delta(0;\rho)$, where each $b_{\nu_{k,\ell}} \in \mathbb{K}$ arises as a finite sum of products of powers of the coefficients $a_{\theta_k}$.

Reordering the summation and regrouping terms, we obtain
\[
\sum_{\theta_k \in \N_0^m} A_{\theta_k} z^{\theta_k} = 0 \quad \text{for all } z \in \Delta(0; \rho),
\]
with
\begin{equation}\label{coeff}
A_{\theta_k} = \sum b_{\nu_{i,j}} C_{\sigma_j},
\end{equation}
where the sum runs over all $\sigma_j = (s_{1,j}, \ldots, s_{m+1,j}) \in \N_0^{m+1}$ such that $\nu_{i,j} + (s_{1,j},\ldots,s_{m,j}) = \theta_k$.

Since a power series is identically zero if and only if each of its coefficients equals zero, it follows that \( A_{\theta_k} = 0 \) for all \( \theta_k \in \mathbb{N}_0^m \).
 Let
\[
\mathcal{A}_{\theta_k}(x_1, \ldots, x_{n_P}) = 0
\]
denote the $k$-th linear equation with coefficients in $\mathbb{K}$ obtained by replacing each $C_{\sigma_j}$ with an indeterminate $x_j$ in \eqref{coeff}, where $n_P$ is the number of distinct monomials in $P$.

Among these infinitely many equations, at most $n_P$ are linearly independent. The system admits a nontrivial solution over $\mathbb{C}$, namely the original coefficients $x_j = C_{\sigma_j}$. Since the coefficients of each equation lie in $\mathbb{K}$, the system also admits a nontrivial solution $(\tilde{C}_{\sigma_j})$ with each $\tilde{C}_{\sigma_j} \in \mathbb{K}$. Hence,
\[
0 = \sum_{\sigma_\ell} \tilde{C}_{\sigma_\ell} z_1^{s_{1,\ell}} \cdots z_m^{s_{m,\ell}} \left( \sum_{\theta_k \in \N_0^m} a_{\theta_k} z^{\theta_k} \right)^{s_{m+1,\ell}}
\]
for all $z \in \Delta(0; \rho)$.

Defining
\[
\tilde{P}(X_1,\ldots,X_m,Y) := \sum_{\sigma_\ell} \tilde{C}_{\sigma_\ell} X_1^{s_{1,\ell}} \cdots X_m^{s_{m,\ell}} Y^{s_{m+1,\ell}} \in \mathbb{K}[X_1,\ldots,X_m,Y],
\]
we conclude that
\[
\tilde{P}(z_1,\ldots,z_m, f(z_1,\ldots,z_m)) = 0 \quad \text{for all } z \in \Delta(0; \rho),
\]
as desired.
\end{proof}

As an immediate consequence of the previous lemma, we obtain the following corollary.

\begin{corollary}\label{c1}
Let $m$ be a positive integer, and let $\rho$ be an $m$-tuple in $\R_{>0}^m$. Then the set of power series in $m$ variables with coefficients in a countable subfield of $\mathbb{C}$, having polyradius of convergence $\rho$ and algebraic over $\mathbb{C}(z_1,\ldots,z_m)$, is countable.
\end{corollary}

\section{Proof of the Main Theorem}

Let $m$ be a positive integer, and let $\rho = (\rho_1,\ldots,\rho_m) \in (0,1]^m$ be a fixed polyradius. Let $S \subseteq \A_{\rho}^m$ be a subset of algebraic numbers in the polydisc $\Delta(0;\rho)$ that is closed under complex conjugation and contains the zero vector. We consider the following enumeration of $\A_{\rho}^m$:
\begin{equation}\label{enum3}
\A_{\rho}^m = \{\alpha_0, \alpha_1, \ldots\} \cup \{\overline{\alpha}_0, \overline{\alpha}_1, \ldots\},
\end{equation}
where $\alpha_0 = (0,\ldots,0)$. We will construct uncountably many power series $f \in \Z_{\rho}[[z_1,\ldots,z_m]]$ of the form
\begin{equation}\label{mainf}
f(z) = g(z) + \sum_{k=1}^{\infty} z^{\theta_k} h_k(z) \prod_{j=0}^{k-1} f_j(z),
\end{equation}
where $\theta_k = (t_{1,k},\ldots,t_{m,k}) \in \N_0^m$ for all $k \geq 1$, and $z^{\theta_k} = z_1^{t_{1,k}} \cdots z_m^{t_{m,k}}$. Here, $g(z_1,\ldots,z_m)$ is the power series in $\Z_{\rho}[[z_1,\ldots,z_m]]$ given by Lemma~\ref{lem3}, and $f_j(z_1,\ldots,z_m)$ is the power series in $\Z\{z_1,\ldots,z_m\}$ such that $f_j(z) = 0$ for $z \in \A^m$ if and only if $z \in \{\alpha_j, \overline{\alpha}_j\}$, and
\[
|f_j(z_1,\ldots,z_m)| \leq \frac{\hat{C}_j}{(1-\|(z_1,\ldots,z_m)\|)^4},
\]
where $\hat{C}_j$ is a positive constant depending only on $\alpha_j$, which exists for all $j \geq 0$ by Lemma~\ref{lem2}. The functions $h_k(z)$, as well as the $m$-tuples $\theta_k$, will be defined inductively according to the enumeration~\eqref{enum3}.

To define $\theta_j \in \N_0^m$ and $h_j(z) \in \Z\{z_1,\ldots,z_m\}$, we consider the pair $\{\alpha_j, \overline{\alpha}_j\}$ and aim to ensure that
\[
\{f(\alpha_j), f(\overline{\alpha}_j)\} \subseteq \K_{\alpha_j},
\]
where
\[
\K_{\alpha_j} =
\begin{cases}
\QQ, & \text{if } \alpha_j \in S, \\
\C \setminus \QQ, & \text{if } \alpha_j \notin S,
\end{cases}
\]
for all $j \geq 0$.

We begin by setting $\theta_1 = (0,\ldots,0) \in \N_0^m$, so that $\alpha_1^{\theta_1} = 1$. Since $f_0(\alpha_1) \neq 0$, it is easy to see that there exists $\beta_1 \in \C$ such that
\[
\{g(\alpha_1) + \beta_1 f_0(\alpha_1),\ g(\overline{\alpha}_1) + \overline{\beta}_1 f_0(\overline{\alpha}_1)\} \subseteq \K_{\alpha_1},
\]
where we use the fact that $\QQ$ is closed under complex conjugation. Moreover, when $\alpha_1 \in \mathbb{R}^m$, we may choose $\beta_1 \in \mathbb{R}$. By Lemma~\ref{lem1}, there exists a function $h_1(z) \in \Z\{z_1,\ldots,z_m\}$ with bounded coefficients such that $h_1(\alpha_1) = \beta_1$ and $h_1(\overline{\alpha}_1) = \overline{\beta}_1$. Furthermore, there exists a positive constant $\tilde{C}_1$ such that
\[
|h_1(z_1,\ldots,z_m)| \leq \frac{\tilde{C}_1}{1-\|(z_1,\ldots,z_m)\|}.
\]
By~\eqref{mainf}, we have $f(\alpha_1) = g(\alpha_1) + h_1(\alpha_1) f_0(\alpha_1) \in \K_{\alpha_1}$, and similarly $f(\overline{\alpha}_1) \in \K_{\overline{\alpha}_1}$, independently of the choice of $\theta_j$ and $h_j(z)$ for $j \geq 2$. Thus, $\theta_1$ and $h_1(z)$ are well defined.

Proceeding inductively, since $f_j(\alpha_k) \neq 0$ for all $0 \leq j \leq k-1$ and $f_k(\alpha_k) = 0$, we can choose $\theta_k \in \N_0^m$ and $h_k(z) \in \Z\{z_1,\ldots,z_m\}$ such that $\alpha_k^{\theta_k} \neq 0$ and
\[
f(\alpha_k) = g(\alpha_k) + \sum_{j=1}^k \alpha_k^{\theta_j} h_j(\alpha_k) \prod_{i=0}^{j-1} f_i(\alpha_k) \in \K_{\alpha_k}
\]
for all $k \geq 1$.

Moreover, the sequence $\{\theta_k\}_{k \geq 1} \subseteq \N_0^m$ can be chosen such that
\begin{equation}\label{final}
\sum_{k=1}^{\infty} z^{\theta_k} h_k(z) \prod_{j=0}^{k-1} f_j(z)
\end{equation}
belongs to $\Z\{z_1,\ldots,z_m\}$. Indeed, we estimate:
\[
\left|z^{\theta_k} h_k(z) \prod_{j=0}^{k-1} f_j(z)\right| < \frac{C_k \|z\|^{|\theta_k|}}{(1-\|z\|)^{4k+1}},
\]
where $|\theta_k| := t_{1,k} + \cdots + t_{m,k}$ and $C_k := \hat{C}_1 \cdots \hat{C}_{k-1} \tilde{C}_k$. In particular, if we choose $\theta_k$ such that $|\theta_k| \geq \lceil C_k \rceil + k(4k+1)$, then
\begin{align*}
\left|\sum_{k=1}^{\infty} z^{\theta_k} h_k(z) \prod_{j=0}^{k-1} f_j(z)\right| &\leq \sum_{k=1}^{\infty} C_k \|z\|^{\lceil C_k \rceil} \left( \frac{\|z\|^k}{1-\|z\|} \right)^{4k+1} \\
&\leq \frac{1}{e|\log\|z\||} \sum_{k=1}^{\infty} \left( \frac{\|z\|^k}{1-\|z\|} \right)^{4k+1},
\end{align*}
using the fact that for $0 < \|z\| < 1$, the function $x \mapsto x \|z\|^x$ attains its maximum at $x = |1/\log\|z\||$.

Moreover, we have $\|z\|^k/(1-\|z\|) < 1/2$ for all 
\[
k \geq K(z) := \frac{\log(1-\|z\|) - \log 2}{\log\|z\|}.
\]
Hence,
\begin{align*}
\sum_{k=1}^{\infty} \left( \frac{\|z\|^k}{1-\|z\|} \right)^{4k+1}
&= \sum_{k=1}^{\lfloor K(z) \rfloor} \left( \frac{\|z\|^k}{1-\|z\|} \right)^{4k+1} + \sum_{k > \lfloor K(z) \rfloor} \left( \frac{\|z\|^k}{1-\|z\|} \right)^{4k+1} \\
&\leq \sum_{k=1}^{\lfloor K(z) \rfloor} \left( \frac{\|z\|^k}{1-\|z\|} \right)^{4k+1} + \frac{1}{2} \\
&\leq \frac{K(z)}{(1-\|z\|)^{4K(z)+1}} + \frac{1}{2}.
\end{align*}

Therefore,
\[
\sum_{k=1}^{\infty} \left| z^{\theta_k} h_k(z) \prod_{j=0}^{k-1} f_j(z) \right| \leq \frac{1}{e|\log\|z\||} \left( \frac{K(z)}{(1-\|z\|)^{4K(z)+1}} + \frac{1}{2} \right).
\]
This shows that the series in \eqref{final} belongs to $\Z\{z_1,\ldots,z_m\}$ whenever $|\theta_k| \geq \lceil C_k \rceil + k(4k+1)$. Since the only further requirement on $\theta_k$ is that $\alpha_k^{\theta_k} \neq 0$, we can construct uncountably many such power series as in \eqref{mainf} with polyradius $\rho$ and satisfying $\{f(\alpha_k),f(\overline{\alpha}_k)\} \subseteq \K_{\alpha_k}$ for all $k \geq 0$.

Therefore, by Corollary~\ref{c1}, there exist uncountably many transcendental functions $f \in \Z_{\rho}[[z_1, \ldots, z_m]]$ such that $S_f = S$, which completes the proof. \qed

\begin{remark}
The proof presented in this work not only extends the result proved in \cite{gugu20} to several variables, but also simplifies the argument. In particular, the construction introduced here does not require the use of Liouville-type inequalities between algebraic numbers to establish the transcendence of elements outside the set $S$.
\end{remark}

We conclude this work by proposing a matrix-valued version of Mahler’s Problem {\sc C}. Given a subset of matrices with algebraic entries that is closed under complex conjugation and contains the zero matrix, does there exist an analytic matrix-valued function whose exceptional set coincides with the given subset? In other words, is there an analytic matrix function that maps exactly those matrices with algebraic coefficients in the domain to matrices with algebraic coefficients in the codomain?

\end{document}